\documentclass[11pt]{amsart}
\usepackage{amsmath,amssymb,amsbsy,amsfonts,latexsym,amsopn,amstext,cite,
                                               amsxtra,euscript,amscd,bm}
\usepackage{url}
\usepackage{mathrsfs}
\usepackage{color}
\usepackage[colorlinks,linkcolor=blue,anchorcolor=blue,citecolor=blue,backref=page]{hyperref}
\usepackage{color}
\usepackage{graphics,epsfig}
\usepackage{graphicx}
\usepackage{float}
\usepackage{epstopdf}
\hypersetup{breaklinks=true}

\usepackage[np]{numprint}
\npdecimalsign{\ensuremath{.}}

\usepackage{bibentry}

\usepackage[english]{babel}
\usepackage{mathtools}
\usepackage{todonotes}
\usepackage{url}
\usepackage[colorlinks,linkcolor=blue,anchorcolor=blue,citecolor=blue,backref=page]{hyperref}

 \def \fh{{\mathfrak h}}

\newtheorem{Thm}{Theorem}[section]
\newtheorem{lemma}[Thm]{Lemma}

\newtheorem{definition}[Thm]{Definition}

\newtheorem*{theorem*}{Theorem}
\newtheorem{theorem}{Theorem}

\theoremstyle{definition}
\newtheorem{remark}[Thm]{Remark}

 \numberwithin{equation}{section}

\numberwithin{Thm}{section}

\numberwithin{equation}{section}

\def\({\left(}
\def\){\right)}
\def\[{\left[}
\def\]{\right]}
\def\<{\langle}
\def\>{\rangle}

\def\le{\leqslant}
\def\ge{\geqslant}

\def \bt{t}
\def \bx{\mathbf{x}}

\def\mand{\qquad \mbox{and} \qquad}

\numberwithin{equation}{section}
\numberwithin{table}{section}

\newcommand{\cA}{{\mathcal A}}

\DeclareMathOperator{\Prob}{Prob}

\renewcommand{\L}{{\mathbb L}}
\newcommand{\F}{{\mathbb F}}
\newcommand{\bfa}{{\mathbf a}}
\newcommand{\bfb}{{\mathbf b}}
\newcommand{\bfc}{{\mathbf c}}

\newcommand{\Lbar}{{\overline{\mathbb L}}}
\newcommand{\Fqbar}{\overline{\mathbb{F}_q}}

\newcommand{\Kbar}{\overline{K}}

\newcommand{\bG}{{\mathbb G}}
\newcommand{\bN}{{\mathbb N}}

\newcommand{\bC}{{\mathbb C}}

\newcommand{\bF}{{\mathbb F}}
\newcommand{\Fq}{\bF_q}
\newcommand{\lra}{\longrightarrow}

\newcommand{\cT}{\mathcal{T}}

\begin{document}

\title[Order of torsion]{Order of torsion for reduction of linearly independent points for a family of Drinfeld modules}

\author[D. Ghioca]{Dragos Ghioca}
\author[I. E. Shparlinski]{Igor E. Shparlinski}


\address{
Dragos Ghioca\\
Department of Mathematics\\
University of British Columbia\\
Vancouver, BC V6T 1Z2\\
Canada
}

\email{dghioca@math.ubc.ca}

\address{Igor E. Shparlinski\\ 
School of Mathematics and Statistics\\
University of New South Wales\\
Sydney NSW 2052\\
Australia
}

\email{igor.shparlinski@unsw.edu.au}

\begin{abstract}
Let $q$ be a power of the prime number $p$, let $K=\bF_q(t)$, and let $r\ge 2$ be an integer. For points $\bfa, \bfb\in K$ which are $\bF_q$-linearly independent, we show that there exist positive constants $N_0$ and $c_0$ such that for each integer $\ell\ge N_0$ and for each generator $\tau$ of $\F_{q^\ell}/\F_q$, we have that  for all except $N_0$ values $\lambda\in\Fqbar$, the corresponding specializations $\bfa(\tau)$ and $\bfb(\tau)$ cannot have orders of degrees less than $c_0\log\log\ell$  as torsion points for the Drinfeld module $\Phi^{(\tau,\lambda)}:\Fq[T]\lra {\rm End}_{\Fqbar}(\bG_a)$ (where $\bG_a$ is the additive group scheme), given by $\Phi^{(\tau,\lambda)}_T(x)=\tau x+\lambda x^q + x^{q^r}$. 
\end{abstract}

\maketitle


\section{Introduction}


\subsection{Notation}

Let $q$ a power of  a  prime number $p$. Consequently, we denote by $\Fq$ the finite field with $q$ elements and also, let $\Fqbar$ be its algebraic closure. We let $K=\Fq(t)$ be the rational function field of transcendence degre $1$ over $\Fq$. Then each $\bfa\in K$ can be written as a quotient ${P(t)}/{Q(t)}$ for relatively prime polynomials $P,Q\in\Fq[t]$ (with $Q\ne 0$); so, for all but finitely many $\tau \in\Fqbar$ (as long as $Q(\tau)\ne 0$), the specialization map $t\mapsto \tau$ yields a specialization $\bfa\mapsto \bfa( \tau)\in\Fqbar$. Furthermore, for any polynomial $f\in K[x]$, then for all but finitely many $ \tau\in\Fqbar$ the specialization $t\mapsto  \tau$ is well-defined for each coefficient of $f(x)$; in this case, we denote by $f(\tau ;x)\in\Fqbar[x]$ the corresponding specialization of $f(x)$.

An ($\Fq[T]$-)Drinfeld module $\Phi$ defined over any field $L$ containing $\Fq$ is given by a non-linear action of $\Fq[T]$ on the additive group scheme $\bG_a$, that is, it is an endomorphism $\Phi:\Fq[T]\lra {\rm End}_L(\bG_a)$ where not all of the maps $\Phi_{P(T)}:=\Phi(P(T))$ are linear. In particular, each such Drinfeld module is uniquely defined by an {\it additive polynomial\/} (of degree larger than $1$) with coefficients in $L$ corresponding to $\Phi_T$.  We recall thar  a polynomial $f(x)$ is called \emph{additive} if $f(x+y)=f(x)+f(y)$ for all $x$ and $y$. These are linear polynomials 
in characteristic zero and polynomials with monomial terms of degrees which are powers of $p$ in positive characteristic $p$.  
A point $\bfc\in\overline{L}$ is a torsion point for $\Phi$ if there exists a nonzero polynomial $P(T)\in\Fq[T]$ such that $\Phi_{P(T)}(\bfc)=0$; the  \emph{monic} polynomial $P(T)$ of minimal degree for which $\Phi_{P(T)}(\bfc)=0$ is called the \emph{order} of $\bfc$.  For more details on Drinfeld modules, we refer the reader to~\cite{Goss}.


\subsection{Unlikely intersections}

Inspired by the classical unlikely intersection problems in arithmetic geometry (for example, see~\cite{Habegger-0} and the references therein), Masser and Zannier~\cite{M-Z-1, M-Z-2} have proved a remarkable result: there exist finitely many $\lambda\in \bC$ such that both the points $\left(2,\sqrt{2(2-\lambda)}\right)$ and $\left(3,\sqrt{6(3-\lambda)}\right)$ are torsion on the corresponding elliptic curve $E_\lambda$ from the Legendre family  given by the equation $y^2=x(x-1)(x-\lambda)$. The results from~\cite{M-Z-1, M-Z-2} have been the starting point of a very intense research campaign in arithmetic dynamics, which emulates the classical unlikely intersection problems from arithmetic geometry, but this time, the unlikely intersection taking place in an algebraic dynamics setting (for example, see~\cite{B-D, GHT-ANT, GKN}). In this spirit, the first author and Hsia have proved in~\cite[Theorem~1.5]{GH-Acta} the following result.

\begin{theorem}
\label{thm:unlikely intersections}
Let $r>1$ be an integer, let $q$ be a power of the prime number $p$, and let $K=\Fq(t)$. We consider the family of Drinfeld modules $\Phi^{(z)}$ (parametrized by $z\in\Kbar$) such that 
$$\Phi_T^{(z)}(x)=tx + z x^q + x^{q^r}.$$ 
Let $\bfa,\bfb\in K$. Then there exist infinitely many $z\in \Kbar$ such that both $\bfa$ and $\bfb$ are torsion points for the Drinfeld module $\Phi^{(z)}$ if and only if $\bfa$ and $\bfb$ are $\Fq$-linearly dependent.
\end{theorem}

\begin{remark}
In~\cite[Theorem~1.5]{GH-Acta}, the above conclusion in Theorem~\ref{thm:unlikely intersections} has been stated under the assumption that there exist infinitely many $\lambda$ in the separable closure $K^{{\rm sep}}$ of $K$. On the other hand, we observe that the main technical ingredient used in the proof of \cite[Theorem~1.5]{GH-Acta} comes from the equidistribution theorem of Baker and Rumely~\cite[Theorem~7.52]{BR}, which is stated for product formula fields. So, in~\cite[Theorem~2.6]{GH-Acta} (which provides the key step in proving \cite[Theorem~1.5]{GH-Acta}), one could work instead with the perfect closure $K^{{\rm per}}:=\Fq\left(t^{1/p^n}\colon n\ge 1\right)$, which is a product formula field since each valuation on $K$ has a unique extension to $K^{{\rm per}}$ and thus, the hypotheses (1)-(4) from \cite[Definition~7.51,~p.~185]{BR} are met for $K^{{\rm per}}$. Then since $\left(K^{{\rm per}}\right)^{{\rm sep}}=\Kbar$, one obtains our Theorem~\ref{thm:unlikely intersections} with an identical proof as in \cite[Theorem~1.5]{GH-Acta}. 
\end{remark}

\begin{remark}
\label{rem:strict}
We also note that if $\bfb=\zeta\cdot \bfa$ for some $\zeta\in\Fq$, then $\Phi^{(z)}_{P(T)}(\bfb)=\zeta\cdot \Phi^{(z)}_{P(T)}(\bfa)$ for \emph{each} $z\in \Kbar$; therefore, we have that $\bfa$ is torsion for $\Phi^{(z)}$ if and only if $\bfb$ is torsion for $\Phi^{(z)}$, and moreover, they both have the same order as torsion points for $\Phi^{(z)}$. 
\end{remark}

One of the unlikely intersection results inspired by~\cite{M-Z-1, M-Z-2} is the result of Habegger~\cite{Habegger} that there exist finitely many pairs $(a,b)\in\bC^2$ with the property that \emph{each} of the three points 
\begin{equation}
\label{eq:points elliptic}
\left(1,\sqrt{1+a+b}\right), \quad \left(2,\sqrt{8+2a+b}\right), \quad \left(3,\sqrt{27+3a+b}\right)
\end{equation}
are torsion points for the elliptic curve $y^2=x^3+ax+b$. This motivated the second author to prove the following result (see~\cite[Theorem~1.1]{Igor}). 
\begin{theorem}
\label{thm:elliptic}
There is an absolute constant $M_0$ such that for all primes $p$ and
for all but at most $M_0$ pairs $(a, b)\in \overline{\F_p}^2$ with $4a^3 + 27b^2\ne 0$ at least one of the
points from~\eqref{eq:points elliptic} is of order at least $(\log  p)^{1/33+o(1)}$.
\end{theorem}

Theorem~\ref{thm:elliptic} is proven in~\cite{Igor} by employing both the result of Habegger~\cite{Habegger}, but also using in an essential way~\cite[Theorem~2.1]{TAMS} which provides a way to control the number of solutions to polynomial equations when they are reduced modulo primes. 

\subsection{Our results}

Here  we prove the following result which is a counterpart for Drinfeld modules of Theorem~\ref{thm:elliptic} (also, note that as observed in~\cite[Chapter~3]{Goss}, the family of Drinfeld modules $\Phi^{(z)}$ which plays the counterpart role of elliptic curves is the one for which $\Phi^{(z)}_T(x)=tx+z x^q + x^{q^2}$, as we vary $z$ in $\Kbar$).
\begin{Thm}
\label{thm:main}
Let $r\ge 2$ be an integer, let $q$ be a power of the prime $p$, and let $K=\Fq(t)$. We let $\bfa,\bfb\in K$ be two $\Fq$-linearly independent points. Then there exist positive constants $N_0$ and $c_0$ (depending only on $q$, $r$ and the degrees of $\bfa$ and $\bfb$) such that for each $\ell\ge N_0$ and for each generator $\tau$ of the extension $\F_{q^\ell}/\F_q$, we have that for all except $N_0$ values of $\lambda\in\Fqbar$, at least one of the two points $\bfa(\tau)$ or $\bfb(\tau)$ is a torsion point whose order has degree at least $c_0\log\log(\ell)$ for the Drinfeld module $\Phi^{(\tau,\lambda)}\in {\rm End}_{\Fqbar}(\bG_a)$ given by $\Phi^{(\tau,\lambda)}_T(x):=\tau x+\lambda x^q + x^{q^r}$.
\end{Thm} 

\begin{remark}
In Theorem~\ref{thm:main}, the family of Drinfeld modules $\Phi^{(\tau,\lambda)}$ given by $\Phi^{(\tau,\lambda)}_T(x):=\tau x+\lambda x^q + x^{q^r}$ varies with respect to both parameters 
$\tau$ and $\lambda$ in $\Fqbar$. Note that for any given $\tau\in \Fqbar$, one could easily find $\lambda\in\Fqbar$ such that $\bfa(\tau)$ is killed by $\Phi_T$, say, and therefore $\bfa(\tau)$ could be a torsion point of order of a  very small degree for the Drinfeld module $\Phi^{(\tau,\lambda)}$; but then the conclusion of our Theorem~\ref{thm:main} yields that $\bfb(\tau)$ cannot be a torsion point whose order has also very small degee compared to $[\Fq(\tau):\Fq]$. Again, just like it has been observed in Remark~\ref{rem:strict}, the fact that $\bfa$ and $\bfb$ are $\Fq$-linearly independent is paramount since otherwise, if $\bfb/\bfa\in \Fq^*$ (for example)  then for any suitable  specialization $t\mapsto \tau$, we would have that also $\bfa(\tau)$ and $\bfb(\tau)$ are $\Fq$-linearly dependent and therefore, they are torsion points for any Drinfeld module $\Phi^{(\tau,\lambda)}$ of the same order (which thus could be very small relative to $[\Fq(\tau):\Fq]$).  
\end{remark}

\begin{remark}
The fact that in the conclusion of our Theorem~\ref{thm:main} we ask that  $\ell\ge N_0$ is merely for guaranteeing that the specialization $\bfa\mapsto \bfa(\tau)$ (resp. $\bfb\mapsto \bfb(\tau)$) is well-defined for each generator $\tau$ for the extension $\F_{q^\ell}/\F_q$ (note that there are only finitely many $\tau\in\Fqbar$ for which the above specializations do not make sense).
\end{remark}

\begin{remark}
The extra logarithm appearing in $c_0\log\log\ell$ from the conclusion of our Theorem~\ref{thm:main} comes form the fact that we insist that for \emph{each}  specialization $t\mapsto \tau$ where $\tau$ \emph{generates} the extension $\F_{q^\ell}/\Fq$ our desired conclusion regarding the orders of the torsion points $\bfa(\tau)$ and $\bfb(\tau)$ must hold. Our proof could be easily modified to show that there exists \emph{some} $\tau$ generating the extension $\F_{q^\ell}/\Fq$ with the property that the  specialization $t\mapsto \tau$ yields that at least one of $\bfa(\tau)$ or $\bfb(\tau)$ would be a torsion point for $\Phi^{(\tau,\lambda)}$ (for each $\lambda\in\Fqbar$ apart from some $N_0$ values $\lambda$) whose order has degree at least $c_1\log(\ell)$ (for some positive constant $c_1$ depending only on $q$, $r$ and the degrees of $\bfa$ and $\bfb$).   
\end{remark}

We also note that our Theorem~\ref{thm:main} (similar to  Theorem~\ref{thm:elliptic} of~\cite{Igor}) fits into a more general program of obtaining results for specialization over finite fields for certain global results stemming from arithmetic dynamics (see also~\cite{A-G, BCMOS,Chang-1,   CKSZ, TAMS, KMS, Mello, Shp18}). 
Also, strictly in the context of Drinfeld modules, there has been interest in the past for considering Drinfeld modules defined over finite fields, either for studying some of their intrinsic arithmetic properties (as in~\cite{G-TAMS-0}), or for infering aritmetic properties of Drinfeld modules of generic characteristic after considering their reductions at various places (see~\cite{G-T, Hsia}). Furthermore, our results provide additional evidence that Drinfeld modules are the right vehicle in positive characteristic for formulating questions which are analogue to classical arithmetic problems in the context of abelian varieties (see \cite{Florian, G-IMRN, Scanlon} for results in the context of Drinfeld modules, which are analogues of the classical Andr\'e-Oort, Mordell-Lang, respectvely Manin-Mumford conjectures). 

We also note that~\cite[Theorem~5.2]{A-G} provides a somewhat related result to our Theorem~\ref{thm:main}; however, there are some significant differences between that result and ours. The result from~\cite[Theorem~5.2]{A-G} refers to reductions at various places of a \emph{given} Drinfeld module, while in our result we vary both the specialization $t\mapsto \tau$ but also vary $\lambda$ in $\Fqbar$; in other words, $\left\{\Phi^{(\tau,\lambda)}\right\}$ is a two-dimensional family of Drinfeld modules defined over $\Fqbar$. Finally, we refer to our Section~\ref{sec:comments} for a discussion regarding further research in the direction set by our paper.

We sketch briefly the plan for our paper. The approach to our Theorem~\ref{thm:main} follows the general direction set in~\cite{Igor} for the case of elliptic curves. So, we need to establish a technical result (which is our Lemma~\ref{lem:reduction}) regarding the preservation of the number  of zeros in specializations of  polynomials over a function field. Then using also a couple of technical results (see Lemmas~\ref{lem:one} and~\ref{lem:two}) regarding the growth of heights and of degrees of the iterates of a given point under polynomials in a family of Drinfeld modules, we obtain the desired technical ingredients for finishing the proof of our Theorem~\ref{thm:main}, which is completed  Section~\ref{sec:proofs}. We conclude with a discussion in Section~\ref{sec:comments} of further research in the direction opened up by our article.


\section{Preliminary results}

\subsection{Greatest common divisors of polynomials}

We recall that a greatest common divisor of polynomials  $f_1, \ldots, f_s \in \L[x]$ over a field $\L$ 
is any polynomial $d\in\L[x]$ such that any other common divisor $e$ of $f_1, \ldots, f_s$ divides $d$
(clearly $d$ is defined up to a nonzero scalar multiple). 
In particular,  we can define the degree $\deg \gcd(f_1, \ldots, f_s)$,  which does not 
depend on a particular choice of a greatest common divisor. 

We need the following special case of~\cite[Lemma~4.8]{vzGKS}. 

\begin{lemma}
\label{lem:many2two}
Let $\L$ be a field and let 
$f_1, \ldots, f_s \in \L[x]$ be $s \ge 2$ polynomials of degree
at most $D$. For any finite set  
$\cA \subseteq \L$, 
uniformly chosen random elements $\alpha_3, \ldots, \alpha_s \in \cA$ 
and 
$$
f_0= f_2 + \sum_{3 \le i \le s} \alpha_i f_i \in \L[x].
$$
we have 
$$\Prob_{\alpha_3, \ldots, \alpha_s \in \cA}\left[ \deg \gcd(f_1, \ldots, f_s) =\deg  \gcd(f_0, f_1) \right] \geq 1-  D/\#\cA.$$
\end{lemma}

We can now estimate the degree of the coefficients in the B{\'e}zout identity for 
relatively prime polynomials. 

\begin{lemma}
\label{lem:Bezout}
Let $\L$ be a field and let 
$f_1, \ldots, f_s \in \L[x]$ be $s \ge 2$  relatively prime polynomials of degree
at most $D$. There are polynomials $b_1, \ldots, b_s \in \L[x]$ of degree
at most $D$ such that  
\begin{equation}
\label{eq:Bezout}
b_1f_1 + \ldots + b_s f_s = 1.
\end{equation}
\end{lemma}

\begin{proof} Taking the algebraic closure $\Lbar$ of $\L$ (which is not necessary if $\L$ is infinite, or even finite but satisfying  
$\#\L > D$), we see from Lemma~\ref{lem:many2two}
that there are some $\alpha_3, \ldots, \alpha_s \in \Lbar$ such that $f_1$ and $f_0$, defined as in Lemma~\ref{lem:many2two} are relatively prime. 
 Note that $\deg f_0\le D$ and hence using the standard Euclidean algorithm we see that 
there are  $b_0, b_1 \in \Lbar[x]$ of degrees at most $D-1$  such that 
$$
b_0 f_0 + b_1f_1 =1
$$
and we obtain~\eqref{eq:Bezout} with $b_1, b_2 = b_0, b_3= \alpha_3 b_0, \ldots, b_s =\alpha_s b_0 \in \Lbar [x]$.

We now observe that~\eqref{eq:Bezout} defines a system of linear equations for $Ds$ coefficients of 
polynomials $b_1, \ldots, b_s$. The above argument shows that that system is compatible over $\Lbar$,
and since the coefficients of this system are defined over $\L$, then this means that the system is compatible over $\L$ as well.
Hence we can find  polynomials $b_1, \ldots, b_s \in \L[x]$ satisfying~\eqref{eq:Bezout}. 
\end{proof} 

%

%
%
%
%
%


\subsection{Preserving number of common zeros in specializations of  polynomials over a function field}



\begin{definition}
For a polynomial $f\in \Fq[t][x]$, we define  its height $\fh(f)$ as the maximum degree of the coefficients of $f$, which are themselves polynomials in $\Fq[t]$. (Note that for any element in $\Fq(t)$, its Weil height is simply its degree as a rational function.)
\end{definition} 

One of our main tools is the following result. Note that in our proof of Theorem~\ref{thm:main}, we only employ Lemma~\ref{lem:reduction} for two polynomials (that is, $s=2$ with the notation as in Lemma~\ref{lem:reduction}); however, we include this more general version of Lemma~\ref{lem:reduction} which may be useful beyond our current application.
For further generalisations see Section~\ref{sec:comments}. 

\begin{lemma}
\label{lem:reduction}
Let $f_1,\ldots, f_s\in \Fq[t][x]$ be polynomials of degree at most $D$ and of heights at most $H$. Assume that there exist at most $N_0$ common zeros in $\overline{\Fq(t)}$ of the polynomials $f_1(x), \ldots , f_s(x)$.
Then there exists a nonzero polynomial $\cT\in\Fq[T]$ of degree at most $(2D+1)H$
 with the property that for each $\tau \in\Fqbar$ such that $\cT(\tau)\ne 0$,  for  the  specialization polynomials 
 $f_i(x)\mapsto f_i(\tau; x)$, $i=1, \ldots, s$, corresponding to the specialization $t\mapsto \tau$,  we have that  the polynomials  
$$f_1(\tau; x), \ldots,  f_s(\tau;x)\in \Fqbar[x]$$ 
have at most $N_0$ common zeros in $\Fqbar$. 
\end{lemma}

\begin{proof}
Let $h \in \Fq[t][x]$ be the greatest primitive common divisor of $f_1,\ldots, f_s$, 
that is, the coefficients of $h$ are relatively prime as polynomials in $t$. 
We first claim that considering $f_1,\ldots, f_s, h$ as bivariate polynomials in $\Fq[t,x]$ 
we still have the divisibility  $h \mid f_i$, $i =1,\ldots, s$. 
Indeed, we know that $h$ divides $f_i$, $i =1,\ldots, s$, 
in the ring $\Fq(t)[x]$ and so there are $g_i \in \Fq[t][x]$ and $b \in \Fq[t]$ 
such that $b$ is relatively prime to $g_1, \ldots, g_s$ as polynomials in $\Fq[t,x]$ and  
$$
f_i = (g_i/b)\cdot h
$$
that is 
$$
bf_i =  hg_i
$$
Since $h$ is primitive and  $\Fq[t,x]$ is a unique factorisation domain, we deduce that $b\in \F_q$. 

Hence for $g_i := f_i/h$ we have  $g_i \in \Fq[t][x]$, $i =1,\ldots, s$. 
Furthermore, we see that $g_1, \ldots, g_s$ have no common zeros in $\overline{\Fq(t)}$  
and hence by Lemma~\ref{lem:Bezout} there are for some $b_1, \ldots, b_s \in \Fq(t)[x]$
of degree at most $D-1$ such that 
\begin{equation}
\label{eq:bigi}
b_1g_1+\ldots+b_sg_s=1. 
\end{equation}

We now see that~\eqref{eq:bigi} can be rewritten as a system of at most $2D$ linear equations over $\Fq(t)$
in $Ds$ variables. Choosing that largest nonsingular submatrix of the matrix defining this system, from the Cramer rule we see that~\eqref{eq:bigi} has a solution in $b_i = c_i/\Delta$ where
$c_i \in \Fq[t][x]$, $i=1, \ldots, s$ and  $\Delta \in \Fq[t]$ is a polynomial of degree (in $t$)  
$$
\deg(\Delta) \le 2DH.
$$
We now write~\eqref{eq:bigi}  as $c_1g_1+\ldots+c_sg_s=\Delta$ or 
\begin{equation}
\label{eq:cigid}
c_1f_1+\ldots+c_sf_s=\Delta h.
\end{equation}
Therefore, for any $\tau \in \Fqbar$ which is not a root of $\Delta(t)$ we derive from~\eqref{eq:cigid} that  
$$
c_1(\tau; x) f_1(\tau; x)+\ldots+c_s(\tau; x)f_s(\tau; x) =\Delta(\tau) h(\tau; x).
$$

By our assumption, $h(x)$ has at most $N_0$ roots in $\overline{\Fq(t)}$; hence, 
unless  $h(\tau; x)$ is identically equal to zero, then it must have at most $N_0$ roots in $\Fqbar$.

Since as we have observed  the bivariate polynomial $h\in \Fq[t,x]$ divides, for example, $f_1\in\Fq[t,x]$, then 
we see that $\fh(h) \le H$. In particular, the greatest monic common divisor $h_{{\rm gcd}}(t)\in\Fq[t]$ of all the coefficients of $h(x)$ (which are themselves polynomials in $\Fq[t]$) must have degree at most $H$.  Hence,  letting $\cT(t):=\Delta(t)\cdot h_{{\rm gcd}}(t)$, we obtain a nonzero polynomial in $\Fq[t]$ of degree at most $2DH+H=(2D+1)H$ which satisfies the desired conclusion from Lemma~\ref{lem:reduction}.  
\end{proof}


\subsection{Technical background on Drinfeld modules}
\label{subsec:Drinfeld}

In this Section we derive some basic results regarding the growth of the heights of the polynomials in a family of Drinfeld modules; for more results regarding heights for Drinfeld modules, we refer the reader to~\cite{Denis, G-JNT, GH-Acta, G-T}.

Let $r\ge 2$ be an integer, and let $K:=\Fq(t)$. We consider the family of Drinfeld modules $\Phi^{(z)}:\Fq[T]\lra {\rm End}_{K(z)}(\bG_a)$ given by $\Phi^{(z)}_T(x):=tx+z x^q+x^{q^r}$ parametrized by $z\in \Kbar$. 

Let $\bfa\in K$ be a nonzero point; we write $\bfa={\bfa_1}/{\bfa_2}$ with (nonzero) relatively prime polynomials $\bfa_1,\bfa_2\in \Fq[t]$. Let $D:=\deg \bfa$ (or equivalently, $D$ is the height of $\bfa$, as defined in~\cite{G-JNT}), that is, $D:=\max\{\deg_t \bfa_1,\deg_t \bfa_2\}$. Then for each positive integer $n$, we have that $\Phi_{T^n}^{(z)}(\bfa)$ is a polynomial in  $K[z]$; for more details, see \cite[Section~4]{GH-Acta}. 

More precisely, let 
$$
f_n(z):=\Phi_{T^n}^{(z)}(\bfa),
$$
then as proven in \cite[Lemma~4.2]{GH-Acta}, we have that $\deg_z(f_n(z))=q^{r(n-1)}$ (since $f_1(z)=\bfa^q \cdot z + t\bfa+\bfa^{q^r}$ is a linear polynomial in $z$). Also, the leading coefficient of $f_n(z)$ is $\bfa^{q^{1+r(n-1)}}$. Furthermore, for each $P(T)\in \Fq[T]$, we may consider the polynomial in $z$ given by $z\mapsto \Phi^{(z)}_{P(T)}(\bfa)$, which is a polynomial with coefficients in $K$. In particular, we have the following immediate corollary of \cite[Lemma~4.2]{GH-Acta}.

\begin{lemma}
\label{lem:two}
Let $0\ne \bfa\in K$ and let $P(T)\in\Fq[T]$ be a non-constant polynomial. Then the map $z\mapsto \Phi^{(z)}_{P(T)}(\bfa)$ is a polynomial in $K[z]$ of degree $q^{r(\deg(T)-1)}$.  
\end{lemma}

 In our proof of Theorem~\ref{thm:main} we also need the following result, 
 which gives us control on the heights of the coefficients of $f_n(z)$.

\begin{lemma}
\label{lem:one 0}
Let $\bfa:=\bfa_1/\bfa_2\in K=\Fq(t)$, where $\bfa_1$ and $\bfa_2$ are nonzero coprime polynomials in $\Fq[t]$. As before, we let $f_n(z):=\Phi_{T^n}^{(z)}(\bfa)\in K[z]$ for each positive integer $n$. Then $g_n(z):=\bfa_2^{q^{rn}}\cdot f_n(z)\in \Fq[t][z]$ and moreover, $$\fh(g_n)\le (2+\deg(\bfa))^{q^{rn}}.$$  
\end{lemma}

\begin{proof}
We have that $f_1(z)=\bfa^qz+t\bfa+\bfa^{q^r}$ and so,
$$g_1(z)=\bfa_2^{q^r}\cdot f_1(z)=\bfa_1^q\bfa_2^{q^r-q}z+ t\bfa_1\bfa_2^{q^r-1} + \bfa_1^{q^r}\in \Fq[t][z];$$
furthermore, $\fh(g_1)\le 1+\deg(\bfa)^{q^r}<(2+\deg(\bfa))^{q^r}$. Also, we have that for each positive integer $n$, 
$$f_{n+1}(z)=f_n(z)^{q^r}+zf_n(z)^q +tf_n(z)$$
and so, 
\begin{equation}
\label{eq:last term}
\begin{split}
g_{n+1}(z)& =\bfa_2^{q^{r(n+1)}}f_{n+1}(z)\\
& = g_n(z)^{q^r} + \bfa_2^{q^{r(n+1)}-q^{1+rn}}zg_n(z)^q + t\bfa_2^{q^{r(n+1)}-q^{rn}}g_n(z).
\end{split} 
\end{equation}
Therefore, an easy induction yields that $g_{n}(z)\in \Fq[t][z]$ and furthermore, $\fh(g_{n}(z))\le (2+\deg(\bfa))^{rn}$ for all $n\in\mathbb{N}$; note that since $\deg(\bfa_2)\le \deg(\bfa)$, then the contribution to the height of $g_{n+1}(z)$ of the last term in \eqref{eq:last term} is at most equal to
$$1+\deg(\bfa)^{q^{r(n+1)}-q^{rn}}\cdot (2+\deg(\bfa))^{q^{rn}}\le (2+\deg(\bfa))^{q^{r(n+1)}}.$$
This concludes our proof of Lemma~\ref{lem:one 0}.
\end{proof}

Finally, the next result is an immediate consequence of Lemma~\ref{lem:one 0}.

\begin{lemma}
\label{lem:one}
Let $\bfa:=\bfa_1/\bfa_2$ for some nonzero coprime polynomials $\bfa_1,\bfa_2\in \Fq[t]$. Then for each non-constant polynomial $P(T)\in \Fq[T]$, the map 
$$z\mapsto \bfa_2^{q^{r\deg_T(P(T))}}\cdot \Phi^{(z)}_{P(T)}(\bfa)$$
is a polynomial in $\Fq[t][z]$ of height at most $(2+\deg(\bfa))^{q^{r\deg_T(P(T))}}$.  
\end{lemma}


\section{Proof of Theorem~\ref{thm:main}}
\label{sec:proofs}

\subsection{Bounds on degree and height of some polynomials} 
So, we have $\bfa,\bfb\in K=\Fq(t)$ which are $\Fq$-linearly independent, and we write 
$$\bfa:=\frac{\bfa_1}{\bfa_2} \mand \bfb:=\frac{\bfb_1}{\bfb_2}
$$ 
with nonzero relatively prime polynomials $\bfa_1,\bfa_2\in\Fq[t]$ (resp. $\bfb_1,\bfb_2\in \Fq[t]$). Then  $D:=\max\{\deg(\bfa),\deg(\bfb)\}$ is a positive integer since otherwise, we would have that both $\bfa$ and $\bfb$ are elements of $\Fq$ and thus, they would be $\Fq$-linearly dependent, contradicting our hypothesis.

According to Theorem~\ref{thm:unlikely intersections}, there exist at most $N_0$ points $z\in \Kbar$ such that both $\bfa$ and $\bfb$ are torsion points for the Drinfeld module $\Phi^{(z)}:\Fq[T]\lra {\rm End}_{K(z)}(\bG_a)$ given by $\Phi^{(z)}_T(x)=tx+z x^q+x^{q^r}$. At the expense of replacing $N_0$ by a larger positive integer, we may also assume that neither $\bfa_2$ nor $\bfb_2$ has a root $\tau$ generating an extension of $\Fq$ of degree at least equal to $N_0$.

Now, let $\ell\ge N_0$ and let $\tau$ be a generator for the finite fields extension $\F_{q^\ell}/\F_q$. Then we can specialize both $\bfa$ and $\bfb$ as $t\mapsto \tau$ (since $\bfa_2(\tau)$ and $\bfb_2(\tau)$ are both nonzero). 

Let $\lambda_0\in \Fqbar$. Then both $\bfa(\tau)$ and $\bfb(\tau)$ are torsion points for the Drinfeld module $\Phi^{(\tau,\lambda_0)}:\Fq[T]\lra {\rm End}_{\Fqbar}(\bG_a)$ given by 
$$
\Phi^{(\tau,\lambda_0)}_T(x)=\tau x+\lambda_0 x^q + x^{q^r}
$$ 
since each element in $\Fqbar$ would be torsion for the Drinfeld module $\Phi^{(\tau,\lambda_0)}$ (which is itself defined over $\Fqbar$). 
Now, assume their respective orders $P(T)$ and $Q(T)$ are both of degree at most $M$ (in $T$) for some real number $M>1$. 

For each non-constant polynomial $R(T)\in\Fq[T]$, we let
\begin{align*}
& g_{\bfa, R(T)}(z):=\bfa_2^{q^{r\deg_T R}}\cdot \Phi^{(z)}_{R(T)}(\bfa), \\
& g_{\bfb, R(T)}(z):=\bfb_2^{q^{r\deg_T R}}\cdot \Phi^{(z)}_{R(T)}(\bfb).
\end{align*}
Then 
$$g_{\bfa,R(T)}(z),g_{\bfb,R(T)}(z)\in\Fq[t][z]
$$ 
are polynomials of height at most $(2+D)^{q^{r\deg_T R}}$ (according to Lemma~\ref{lem:one}) and of degree $q^{r(\deg_T(R) -1)}$ (according to Lemma~\ref{lem:two}).    
Then let 
\begin{align*}
& \tilde{g}_{\bfa,M}(z):=\prod_{\substack{1\le \deg_T R\le M\\ R\text{ is monic}}}g_{\bfa, R(T)}(z)\\
& \tilde{g}_{\bfb, M}(z):=\prod_{\substack{1\le \deg_T R\le M\\ R\text{ is monic}}}g_{\bfb,R(T)}(z).
\end{align*}

The following two results are immediate consequences of Lemmas~\ref{lem:two} and~\ref{lem:one}. 

\begin{lemma}
\label{lem:two easy}
With the above notation,  both $\tilde{g}_{\bfa,M}(z)$ and $\tilde{g}_{\bfb,M}(z)$ have degree less than $q^{(M+1)(r+1)}$.
\end{lemma}

\begin{proof} 
For each $s=1,\ldots, M$, we have $q^s$ monic polynomials $R(T)\in \Fq[T]$ of degree $s$. Since for each monic polynomial $R(T)\in\Fq[T]$ of degree $s$, we have   $\deg_z(g_{\bfa,R(T)}(z))=q^{r(s-1)}$ (according to Lemma~\ref{lem:two}), then we obtain the desired conclusion from Lemma~\ref{lem:two easy}. 
\end{proof}

\begin{lemma}
\label{lem:one easy}
With the above notation, we have that $\fh(\tilde{g}_{\bfa,M}),\fh(\tilde{g}_{\bfb, M})\le c_2(2+D)^{q^{(r+1)(M+1)}}$ for some constant $c_2$ depending only on $D$, $q$ and $r$.
\end{lemma}

\begin{proof}
Again using that we have $q^s$ monic polynomials of degree $s$ in $\Fq[T]$ along with Lemma~\ref{lem:one}, we derive that both $\fh(\tilde{g}_{\bfa,M}(z))$ and $\fh(\tilde{g}_{\bfb,M}(z))$ are bounded above by
$$\sum_{s=1}^M q^s\cdot (2+D)^{q^{rs}}\ll (2+D)^{q^{(r+1)(M+1)}},$$
where the implied constant depends only on $D$, $q$ and $r$. This concludes our proof of Lemma~\ref{lem:one easy}.
\end{proof}

\subsection{Concluding the proof} 
As previously noted, due to Theorem~\ref{thm:unlikely intersections} (because $\bfa$ and $\bfb$ are $\Fq$-inearly independent), we know that there exist at most $N_0$ solutions in $\Kbar$ for the system:
$$\tilde{g}_{\bfa,M}(z)=\tilde{g}_{\bfb,M}(z)=0.$$
Then, by Lemma~\ref{lem:reduction} (see also Lemmas~\ref{lem:one easy} and~\ref{lem:two easy}), we know that there exists a nonzero polynomial $\cT \in\Fq[t]$ of degree at most 
\begin{equation}
\label{eq:the bound}
\left(2q^{(M+1)(r+1)}+1\right)\cdot c_2(2+D)^{q^{(M+1)(r+1)}}<c_3(5+2D)^{q^{(M+1)(r+1)}},
\end{equation}
(for another constant $c_3$ depending only on $D$, $q$ and $r$) such that for each $\tau\in\Fqbar$ such that $\cT(\tau)\ne 0$, then the specialization $t\mapsto \tau$ yields polynomials $\tilde{g}_{\bfa,M}(\tau;x),\tilde{g}_{\bfb,M}(\tau;x)\in \Fqbar[x]$ which have at most $N_0$ common roots $\lambda\in\Fqbar$. So, as long as $\tau$ generates a finite field extension $\F_{q^\ell}/\F_q$ of degree larger than $\deg(\cT)$, then we are guaranteed that the system of equations
\begin{equation}
\label{eq:system reduction}
\tilde{g}_{\bfa,M}(\tau;\lambda)=\tilde{g}_{\bfb,M}(\tau;\lambda)=0
\end{equation}
has at most $N_0$ solutions $\lambda\in\Fqbar$. Hence, we may take $M:=c_0  \log\log(\ell)$ for some  absolute constant $c_0> 0 $, which depends only on $q$, $r$ and $D=\max\{\deg(\bfa),\deg(\bfb)\}$),so that $\ell$ is larger than the bound from \eqref{eq:the bound}. Therefore, in this case, for each $\lambda\in\Fqbar$ other than the $N_0$ common solutions to the system~\eqref{eq:system reduction}, we obtain that not both $\bfa(\tau)$ and $\bfb(\tau)$ are torsion points  under the action of $\Phi^{(\tau,\lambda)}$ of orders of degrees less than $M=c_0\log\log\ell$. This concludes our proof of Theorem~\ref{thm:main}.


\section{Comments}
\label{sec:comments}

In~\cite[Theorem~1.5]{GH-Acta}, a slightly more general result has been proven than the stated Theorem~\ref{thm:unlikely intersections}. More precisely, the same conclusion as in Theorem~\ref{thm:unlikely intersections} holds if $\bfa$ and $\bfb$ are $\Fq$-linearly independent elements of $\Fqbar(t)$. The exact same proof that we used for our Theorem~\ref{thm:main} could be applied for $\bfa,\bfb\in \F_{q^s}(t)$ (for some given $s\in\bN$); the only difference is  that in Theorem~\ref{thm:unlikely intersections} we would need to replace $\Fq$ by $\F_{q^s}$ and $\F_{q^\ell}$ by $\F_{q^{s\ell}}$.   

Furthermore, in~\cite[Theorem~1.4]{GH-Acta}, a similar result as Theorem~\ref{thm:unlikely intersections} is  proven for the same family of Drinfeld modules but replacing $\Fq(t)$ with any function field $K$ in which $\Fq$ is algebraically closed.  Once again, a similar result as the one from Theorem~\ref{thm:main} can be derived when starting with two $\Fq$-linearly independent points $\bfa,\bfb\in K$. Indeed, for all but finitely many specializations $t\mapsto \tau$ (where $\tau\in\Fqbar$) we have well defined specializations $\bfa\mapsto \bfa(\tau)$ and $\bfb\mapsto \bfb(\tau)$ and the rest of the argument follows after replacing the technical statements from Subsection~\ref{subsec:Drinfeld} with similar results inferring heights (for more details regarding heights for Drinfeld modules, see~\cite{G-JNT}). Now, in order to extend the result of our Theorem~\ref{thm:main} to arbitrary $\Fq$-linearly independent points in $\overline{\Fq(t)}$ (or more generally, to arbitrary one-dimensional families of Drinfeld modules), the same strategy from this paper applies as along as one would find a strengthening 
of~\cite[Theorems~1.4 and~1.5]{GH-Acta}; however, this last part seems very difficult  at this moment since the conclusions in~\cite[Theorems~1.4 and~1.5]{GH-Acta} are obtained through some nontrivial arguments using valuation theory, which do not seem amenable to generalizations.

Finally, one could consider a higher-dimensional variant of our Theorem~\ref{thm:main} by studying a family of Drinfeld modules parametrized by more than one variable. So, for example, letting once again $K:=\Fq(t)$, a natural question would be to consider the $2$-dimensional family of Drinfeld modules $\Psi^{(z_1,z_2)}:\Fq[T]\lra {\rm End}_{K(z_1,z_2)}(\bG_a)$ given by
$$\Psi^{(z_1,z_2)}_T(x):=tx + z_1x^q+z_2 x^{q^2} + x^{q^3}$$
along with three $\Fq$-linearly independent points $\bfa,\bfb,\bfc\in K$. Then the \emph{expectation} is that there are only finitely many pairs of $(z_1,z_2)\in \Kbar^2$ such that $\bfa,\bfb,\bfc$ are simultaneously torsion points for the corresponding Drinfeld module $\Psi^{(z_1,z_2)}$. Even though proving this expectation is beyond the known methods in arithmetic dynamics, once such a result would be established, then our current proof strategy would still apply verbatim to this more general setting using a suitable multi-variables  generalization of  Lemma~\ref{lem:reduction} (note that Lemmas~\ref{lem:two} and~\ref{lem:one} easily extend to higher-dimensional families of Drinfeld modules).   

There are two possible approaches to such generalizing Lemma~\ref{lem:reduction}  to the multivariate setting of polynomials 
$f_1,\ldots, f_s\in \Fq[\bt][\bx]$  where    $\bx = (x_1, \ldots, x_n)$. 
The first approach, which is perhaps rather elaborate,  is to follow the argument of~\cite{TAMS}  
and establish an analogue of~\cite[Theorem~2.1]{TAMS}, in the function field case. In this case one can use the 
machinery of heights in  function fields,  and for example, rely on 
the {\it effective Hilbert Nullstellensatz\/} given by~\cite[Theorem~5]{DKS} (rather than on its number field 
counterpart~\cite[Theorem~5]{DKS} as in~\cite{TAMS}). The second way is a little simpler and shorter, but may lead to weaker results. 
For this we observe that if the system of equations 
$$
f_1(\bx) = \ldots = f_s(\bx) = 0
$$
has at most $N_0$ solutions in $\bx \in \overline{\Fq(\bt)}^n$, then the polynomial 
$$
F\(\bx_1, \ldots,  \bx_{N_0+1}\) := \prod_{j=1}^{N_0} \(x_{j,1} - x_{N_0,1}\)
$$
vanishes identically on the variety 
\begin{equation}
\label{eq:var}
f_1(\bx_j) = \ldots = f_s(\bx_j ) = 0, \qquad j =1, \ldots, N_0+1.
\end{equation}
Hence the $s(N_0+1)+1$ polynomials $f_i(\bx_j)$, $i =1, \ldots s$, $j =1, \ldots, N_0+1$ 
and $1- z F\(\bx_1, \ldots,  \bx_{N_0+1}\)$ in $n(N_0+1)+1$ variables do not have common zeros.
Now we are under the conditions of~\cite[Theorem~5]{DKS} which leads us to a
multivariate analogue of the relation~\eqref{eq:bigi} and thus allows us to estimate the number 
of specialisations $\bt \mapsto \tau \in \overline \Fq$ for which the corresponding polynomials 
$f_1(\tau;\bx) , \ldots,  f_s(\tau;\bx) $ have more that $N_0$ common zeros.



\section*{Acknowledgements}

The authors are grateful to Carlos D'Andrea and Martin Sombra for very useful discussions
and suggestions. 

D.~G. was partially supported by a Discovery Grant from NSERC, 
 and I.~S. was partially supported by an ARC Grant~DP200100355.


\end{document}